\documentclass[11pt, reqno]{amsart} \textwidth=14.9cm \oddsidemargin=1cm \evensidemargin=1cm
\usepackage{amsmath,amstext,amsbsy,amssymb,amscd}
\usepackage{amsmath}
\usepackage{amsxtra}
\usepackage{amscd}
\usepackage{amsthm}
\usepackage{amsfonts}
\usepackage{amssymb}
\usepackage{eucal}
\usepackage{color}
\usepackage{graphicx}%
\usepackage{bm}
\newtheorem{theorem}{Theorem}[section]
\newtheorem{lemma}[theorem]{Lemma}
\newtheorem{proposition}[theorem]{Proposition}

\theoremstyle{definition}
\newtheorem{definition}[theorem]{Definition}
\newtheorem{remark}[theorem]{Remark}

\definecolor{A}{rgb}{.75,1,.75}

\numberwithin{equation}{section}

\begin{document}

\title[calibrated representations]{Calibrated representations of affine Yokonuma-Hecke algebras}
\author[Weideng Cui]{Weideng Cui}
\address{School of Mathematics, Shandong University, Jinan, Shandong 250100, P.R. China.}
\email{cwdeng@amss.ac.cn}

\begin{abstract}
Inspired by the work [Ra1], we directly give a complete classification of irreducible calibrated representations of affine Yokonuma-Hecke algebras $\widehat{Y}_{r,n}(q)$ over $\mathbb{C},$ which are indexed by $r$-tuples of placed skew shapes. We then develop several applications of this result. In the appendix, inspired by [Ru], we classify and construct irreducible completely splittable representations of degenerate affine Yokonuma-Hecke algebras $D_{r,n}$ and the wreath product $(\mathbb{Z}/r\mathbb{Z})\wr \mathfrak{S}_{n}$ over an algebraically closed field of characteristic $p> 0$ such that $p$ does not divide $r$.
\end{abstract}



\maketitle
\medskip
\section{Introduction}
\subsection{}
Motivated by [Ma] on the dimension of the irreducible $\mathbb{F}\mathfrak{S}_{n}$-modules over an algebraically closed field $\mathbb{F}$ of characteristic $p$, Kleshchev [Kle] studied a class of $\mathbb{F}\mathfrak{S}_{n}$-modules, called completely splittable, on which the Jucys-Murphy elements act semisimply. Ruff [Ru] classified and constructed the irreducible completely splittable representations of degenerate affine Hecke algebras, also recovering Kleshchev's work. Later on, Wan [Wa] classified and constructed irreducible completely splittable representations of affine and finite Hecke-Clifford algebras over an algebraically closed field of characteristic not equal to 2.

Similar objects have arisen in many other contexts and might be used in different terminology, for example Gelfand-Zetlin [Ch, OV], seminormal [Mat], homogeneous [KleRa]. In particular, Ram [Ra1] generalized the classical construction of A. Young to classify and construct all finite dimensional irreducible calibrated representations of affine Hecke algebras of type $A$, and generalized the construction to other types in [Ra2]. This is the class of representations of affine Hecke algebras for which there is a good theory of Young tableaux.

\subsection{}
Yokonuma-Hecke algebras were introduced by Yokonuma [Yo] as a centralizer algebra associated to the permutation representation of a finite Chevalley group $G$ with respect to a maximal unipotent subgroup of $G$. The Yokonuma-Hecke algebra $\mathrm{Y}_{r,n}(q)$ (of type $A$) is a quotient of the group algebra of the modular framed braid group $(\mathbb{Z}/r\mathbb{Z})\wr B_{n},$ where $B_{n}$ is the braid group on $n$ strands (of type $A$). By the presentation given by Juyumaya and Kannan [Ju1, Ju2, JuK], the Yokonuma-Hecke algebra $\mathrm{Y}_{r,n}(q)$ can also be regraded as a deformation of the group algebra of the complex reflection group $G(r,1,n),$ which is isomorphic to the wreath product $(\mathbb{Z}/r\mathbb{Z})\wr \mathfrak{S}_{n}$.

Recently, by generalizing the approach of Okounkov-Vershik \cite{OV} on the representation theory of $\mathfrak{S}_n$, Chlouveraki and Poulain d'Andecy [ChPA1] introduced the notion of the affine Yokonuma-Hecke algebra $\widehat{\mathrm{Y}}_{r,n}(q)$ and gave explicit formulae for all irreducible representations of $\mathrm{Y}_{r,n}(q)$ over $\mathbb{C}(q)$, and obtained a semisimplicity criterion for it. In their subsequent paper [ChPA2], they studied the representation theory of the affine Yokonuma-Hecke algebra $\widehat{\mathrm{Y}}_{r,n}(q)$ and the cyclotomic Yokonuma-Hecke algebra $\mathrm{Y}_{r,n}^{d}(q)$. In particular, they gave the classification of irreducible representations of $\mathrm{Y}_{r,n}^{d}(q)$ in the generic semisimple case. We [CWa] gave the classification of the simple $\widehat{\mathrm{Y}}_{r,n}(q)$-modules as well as the classification of the simple modules of the cyclotomic Yokonuma-Hecke algebras over an algebraically closed field $\mathbb{K}$ of characteristic $p$ such that $p$ does not divide $r.$ Rostam [Ro] proved that the cyclotomic Yokonuma-Hecke algebra is a particular case of cyclotomic quiver Hecke algebras. In the past several years, the study of affine and cyclotomic Yokonuma-Hecke algebras has made substantial progress; see [ChPA1-2, ChS, C1-4, CWa, ER, JaPA, Lu, PA2, Ro].

\subsection{} Inspired by the work [Ra1], in this paper we classify and construct irreducible calibrated representations of affine Yokonuma-Hecke algebras $\widehat{Y}_{r,n}(q)$ over $\mathbb{C},$ which are indexed by $r$-tuples of placed skew shapes. We then develop several applications of this result. In the appendix, inspired by [Ru], we give a complete classification of irreducible completely splittable representations of degenerate affine Yokonuma-Hecke algebras $D_{r,n}$ and the wreath product $(\mathbb{Z}/r\mathbb{Z})\wr \mathfrak{S}_{n}$ over an algebraically closed field of positive characteristic $p$ such that $p$ does not divide $r$.

This paper is organized as follows. In Section 2, we recall some definitions and properties. In Section 3, we present a complete classification and construction of irreducible calibrated representations of $\widehat{Y}_{r,n}(q)$ over $\mathbb{C}.$ In Section 4, we develop several applications of this classification result. In Section 5 (Appendix), we classify and construct the irreducible completely splittable representations of degenerate affine Yokonuma-Hecke algebras $D_{r,n}$ and the wreath product $(\mathbb{Z}/r\mathbb{Z})\wr \mathfrak{S}_{n}$ over an algebraically closed field of positive characteristic $p$ such that $p$ does not divide $r$.

\section{The definition and properties of affine Yokonuma-Hecke algebras}
\subsection{The definition of $\widehat{Y}_{r,n}(q)$}
Fix an element $q\in \mathbb{C}^{*}$ which is not a root of unity.

\begin{definition}
The affine Yokonuma-Hecke algebra, denoted by $\widehat{Y}_{r,n}=\widehat{Y}_{r,n}(q)$, is an $\mathbb{C}$-associative algebra generated by the elements $t_{1},\ldots,t_{n},g_{1},\ldots,g_{n-1},X_{1}^{\pm1},$ in which the generators $t_{1},\ldots,t_{n},g_{1},$ $\ldots,g_{n-1}$ satisfy the following relations:
\begin{equation}\label{rel-def-Y1}\begin{array}{rclcl}
g_ig_j\hspace*{-7pt}&=&\hspace*{-7pt}g_jg_i && \mbox{for all $i,j=1,\ldots,n-1$ such that $\vert i-j\vert \geq 2$,}\\[0.1em]
g_ig_{i+1}g_i\hspace*{-7pt}&=&\hspace*{-7pt}g_{i+1}g_ig_{i+1} && \mbox{for all $i=1,\ldots,n-2$,}\\[0.1em]
t_it_j\hspace*{-7pt}&=&\hspace*{-7pt}t_jt_i &&  \mbox{for all $i,j=1,\ldots,n$,}\\[0.1em]
g_it_j\hspace*{-7pt}&=&\hspace*{-7pt}t_{s_i(j)}g_i && \mbox{for all $i=1,\ldots,n-1$ and $j=1,\ldots,n$,}\\[0.1em]
t_i^r\hspace*{-7pt}&=&\hspace*{-7pt}1 && \mbox{for all $i=1,\ldots,n$,}\\[0.2em]
g_{i}^{2}\hspace*{-7pt}&=&\hspace*{-7pt}1+(q-q^{-1})e_{i}g_{i} && \mbox{for all $i=1,\ldots,n-1$,}
\end{array}
\end{equation}
where $s_{i}$ is the transposition $(i,i+1)$ in the symmetric group $\mathfrak{S}_n$ on $n$ letters, and for each $1\leq i\leq n-1$,
$$e_{i} :=\frac{1}{r}\sum\limits_{s=0}^{r-1}t_{i}^{s}t_{i+1}^{-s},$$
together with the following relations concerning the generators $X_{1}^{\pm1}$:
\begin{equation}\label{rel-def-Y2}\begin{array}{rclcl}
X_{1}X_{1}^{-1}\hspace*{-7pt}&=&\hspace*{-7pt}X_{1}^{-1}X_{1}=1,\\[0.1em]
g_{1}X_{1}g_{1}X_{1}\hspace*{-7pt}&=&\hspace*{-7pt}X_{1}g_{1}X_{1}g_{1}, \\[0.1em]
g_{i}X_{1}\hspace*{-7pt}&=&\hspace*{-7pt}X_{1}g_{i} &&  \mbox{for all $i=2,\ldots,n-1$,}\\[0.1em]
t_{j}X_{1}\hspace*{-7pt}&=&\hspace*{-7pt}X_{1}t_{j} && \mbox{for all $j=1,\ldots,n$,}
\end{array}
\end{equation}
\end{definition}

Note that the elements $e_{i}$ are idempotents in $\widehat{Y}_{r,n}$. The elements $g_{i}$ are invertible, with the inverse given by
\begin{equation}\label{inverse}
g_{i}^{-1}=g_{i}-(q-q^{-1})e_{i}\quad\mbox{for~all}~i=1,\ldots,n-1.
\end{equation}

Let $w\in \mathfrak{S}_{n},$ and let $w=s_{i_1}\cdots s_{i_{r}}$ be a reduced expression of $w.$ By Matsumoto's lemma, the element $g_{w} :=g_{i_1}g_{i_2}\cdots g_{i_{r}}$ does not depend on the choice of the reduced expression of $w$.


Let $i, k\in \{1,2,\ldots,n\}$ and set
\begin{equation}
e_{i,k} :=\frac{1}{r}\sum\limits_{s=0}^{r-1}t_{i}^{s}t_{k}^{-s}.
\end{equation}
Note that $e_{i,i}=1,$ $e_{i,k}=e_{k,i},$ and that $e_{i,i+1}=e_{i}.$ It can be easily checked that the following holds:
\begin{equation}\label{egge}
e_{j,k}g_{i}=g_{i}e_{s_{i}(j),s_{i}(k)}\quad\mbox{for $i=1,\ldots,n-1$ and $j,k=1,\ldots,n$}.
\end{equation}
In particular, we have $e_{i}g_{i}=g_{i}e_{i}$ for all $i=1,\ldots,n-1.$

We define inductively elements $X_{2},\ldots,X_{n}$ in $\widehat{Y}_{r,n}$ by
\begin{equation}
X_{i+1} :=g_{i}X_{i}g_{i}\quad\mathrm{for}~i=1,\ldots,n-1.\label{X2n}
\end{equation}
Then it is proved in [ChPA1, Lemma 1] that we have, for any $1\leq i\leq n-1$,
\begin{equation}
g_{i}X_{j}=X_{j}g_{i}\quad\mathrm{for}~j=1,2,\ldots,n~\mathrm{such~that}~j\neq i, i+1.\label{giXj}
\end{equation}
Moreover, by [ChPA1, Proposition 1], we have that the elements $t_{1},\ldots, t_{n}, X_{1},\ldots, X_{n}$ form a commutative family, that is,
\begin{equation}
xy=yx\quad\mathrm{for~any}~x,y\in \{t_{1},\ldots, t_{n}, X_{1},\ldots, X_{n}\}.\label{xyyx}
\end{equation}
We shall often use the following identities (see [ChPA2, Lemma 2.3]): for $1\leq i\leq n-1$,
\begin{equation}\label{gxxg}\begin{array}{rclcl}
g_{i}X_{i}\hspace*{-7pt}&=&\hspace*{-7pt}X_{i+1}g_{i}-(q-q^{-1})e_{i}X_{i+1}, \\[0.3em]
g_{i}X_{i+1}\hspace*{-7pt}&=&\hspace*{-7pt}X_{i}g_{i}+(q-q^{-1})e_{i}X_{i+1}, \\[0.3em]
g_{i}X_{i}^{-1}\hspace*{-7pt}&=&\hspace*{-7pt}X_{i+1}^{-1}g_{i}+(q-q^{-1})e_{i}X_{i}^{-1}, \\[0.3em]
g_{i}X_{i+1}^{-1}\hspace*{-7pt}&=&\hspace*{-7pt}X_{i}^{-1}g_{i}-(q-q^{-1})e_{i}X_{i}^{-1}.
\end{array}
\end{equation}

\subsection{Combinatorics}
$\lambda=(\lambda_{1},\ldots,\lambda_{k})$ is called a partition of $n$ if it is a finite sequence of non-increasing nonnegative integers whose sum is $n.$ We write $\lambda\vdash n$ if $\lambda$ is a partition of $n,$ and we set $|\lambda| :=n$. We shall identify a partition $\lambda$ with a Young diagram, which is the set $$[\lambda] :=\{(i,j)\:|\:i\geq 1~\mathrm{and}~1\leq j\leq \lambda_{i}\}.$$ We shall regard $\lambda$ as a left-justified array of boxes such that there exist $\lambda_{j}$ boxes in the $j$-th row for $j=1,\ldots,k.$

An $r$-partition of $n$ is an ordered $r$-tuple $\bm{\lambda}=(\lambda^{(1)},\lambda^{(2)},\ldots,\lambda^{(r)})$ of partitions $\lambda^{(k)}$ such that $\sum_{k=1}^{r}|\lambda^{(k)}|=n.$ If $\bm{\lambda}$ and $\bm{\mu}$ are two $r$-partitions such that $\mu_{i}^{(s)}\leq \lambda_{i}^{(s)}$ for all $i$ and $1\leq s\leq r,$ we write $\bm{\mu}\subseteq\bm{\lambda}.$ The $r$-skew shape $\bm{\lambda}/\bm{\mu}$ consists of all boxes of $\bm{\lambda}$ which are not in $\bm{\mu}$. The $\lambda^{(i)}/\mu^{(i)}$ ($1\leq i\leq r$) is called the component of $\bm{\lambda}/\bm{\mu}$, which is a union of connected components. We shall label the boxes of an $r$-skew shape $\bm{\lambda}/\bm{\mu}$ along major diagonals from southwest to northeast of the first component, and then along major diagonals from southwest to northeast of the second component, and so on. If a box contains the number $i$, we denote it by $\mathrm{box}_{i}.$ Let $\mathcal{S}_{r,n}$ denote the set of $r$-skew shapes with $n$ boxes.

Let $\bm{\lambda}/\bm{\mu}$ be an $r$-skew shape with $n$ boxes. A standard tableau of shape $\bm{\lambda}/\bm{\mu}$ is a labelling of the boxes in $\bm{\lambda}/\bm{\mu}$ with the numbers $1,2,\ldots,n$ such that the numbers strictly increase from left to right along each row and from top to bottom down each column. Let $\mathrm{Std}(\bm{\lambda}/\bm{\mu})$ denote the set of standard tableaux of shape $\bm{\lambda}/\bm{\mu}.$

For each $\bm{\lambda}/\bm{\mu}\in \mathcal{S}_{r,n},$ a row reading tableau $\mathcal{R}$ of shape $\bm{\lambda}/\bm{\mu}$ is the standard tableau where $1,2,\ldots,n$ appear from left to right across the rows of the first component, beginning with its northeast most connected component, and then across the rows of the second component, also beginning with its northeast most connected component, and so on. A column reading tableau $\mathcal{C}$ of shape $\bm{\lambda}/\bm{\mu}$ is the standard tableau where $1,2,\ldots,n$ appear from left to right across the columns of the first component, beginning with its southwest most connected component, and then across the rows of the second component, also beginning with its southwest most connected component, and so on.

The following lemma can be regarded as a generalization of [BW, Theorem 7.1] and can be proved similarly.
\begin{lemma}\label{lemma-bijections}
For each $\mathcal{T}\in \mathrm{Std}(\bm{\lambda}/\bm{\mu}),$ we denote by $\mathcal{T}(\mathrm{box}_{i})$ the element containing in $\mathrm{box}_{i}$ of $\mathcal{T}.$ Let $n_{i}=|\lambda^{(i)}|-|\mu^{(i)}|$ for each $1\leq i\leq r.$ We define a subgroup $w_{\mathcal{T}} :=w_{\mathcal{T}_{1}}\times w_{\mathcal{T}_{2}}\times\cdots\times w_{\mathcal{T}_{r}},$ where
\[
w_{\mathcal{T}_{i}}=\left(
  \begin{array}{cccc}
    n_1+\cdots+n_{i-1}+1 & n_1+\cdots+n_{i-1}+2 & \cdots & n_1+\cdots+n_{i-1}+n_{i}\\
    \mathcal{T}(\mathrm{box}_{(n_1+\cdots+n_{i-1}+1)}) & \mathcal{T}(\mathrm{box}_{(n_1+\cdots+n_{i-1}+2)}) &\cdots & \mathcal{T}(\mathrm{box}_{(n_1+\cdots+n_{i-1}+n_{i})})
\end{array}
\right).
\]
Let $\mathcal{C}$ and $\mathcal{R}$ be the column reading and row reading tableaux of shape $\bm{\lambda}/\bm{\mu}.$ The map
\[\mathrm{Std}(\bm{\lambda}/\bm{\mu})\rightarrow \mathfrak{S}_{n}, \qquad \mathcal{T}\mapsto w_{\mathcal{T}}\]
defines a bijection between $\mathrm{Std}(\bm{\lambda}/\bm{\mu})$ and the interval $[w_{\mathcal{C}}, w_{\mathcal{R}}]$ in $\mathfrak{S}_{n}$ (in the Bruhat ordr).
\end{lemma}

Next we give the definition of placed $r$-skew shapes, generalizing the constructions in [Ra1]. Set $\mathbb{R}+i[0,2\pi/\ln(q^2))=\{ a+bi\ |\ a\in \mathbb{R}, 0\le b\le 2\pi/\ln(q^2)
\}\subseteq \mathbb{C}$.  If $q$ is a positive real number then
the following map
\[
  \begin{array}{ccc}
\mathbb{R}+i[0,2\pi/\ln(q^2)) &\longrightarrow &\mathbb{C}^{*} \\
x &\longmapsto &q^{2x}=e^{\ln(q^2)x}
\end{array}
\]
is a bijection.

A pair $(\mathrm{c},\bm{\lambda}/\bm{\mu})$ is called a placed $r$-skew shape if it consists of a skew shape $\bm{\lambda}/\bm{\mu}$ and a content
function
$$\mathrm{c}\colon \{\hbox{boxes of $\bm{\lambda}/\bm{\mu}$}\} \longrightarrow
\mathbb{R}+i[0,2\pi/\ln(q^2))
\qquad\hbox{such that}$$

\[
  \begin{array}{cc}
\mathrm{c}({\rm box}_j)\ge \mathrm{c}({\rm box}_i), \hfill
\quad &\hbox{if $j>i$ and $\mathrm{c}({\rm box}_j)-\mathrm{c}({\rm box}_i)\in \mathbb{Z}$},\\
\mathrm{c}({\rm box}_j)=\mathrm{c}({\rm box}_i)+1,\hfill
&\hbox{if and only if ${\rm box}_j$ and ${\rm box}_i$
are on adjacent diagonals,} \\
\mathrm{c}({\rm box}_j)=\mathrm{c}({\rm box}_i),\hfill
&\hbox{if and only if ${\rm box}_j$ and ${\rm box}_i$
are on the same diagonal.}
\end{array}
\]
This is a generalization of the usual notion of the content of a box
in an $r$-partition.

The following lemma can be regarded as a generalization of [Ra1, Lemma 2.2] and can be proved similarly.
\begin{lemma}\label{lemma-uniques}
Let $(\mathrm{c},\bm{\lambda}/\bm{\mu})$ be a placed $r$-skew shape with $n$ boxes and let $\mathcal{T}$ be a standard tableau of shape $\bm{\lambda}/\bm{\mu}.$ We denote by $\mathcal{T}|i$ the box of $\mathcal{T}$ containing $i$ and $\mathrm{p}(\mathcal{T}|i)$ the index of the component in which the box is. Then the sequence
\[(\mathrm{p}(\mathcal{T}|1),\ldots,\mathrm{p}(\mathcal{T}|n),\mathrm{c}(\mathcal{T}|1),\ldots,\mathrm{c}(\mathcal{T}|n))\]
uniquely determines the placed $r$-skew shape $(c,\bm{\lambda}/\bm{\mu})$ and the standard tableau $\mathcal{T}.$
\end{lemma}

\subsection{The $\tau$ operators}
Let $J=\{1,2,\ldots,r\}.$ We now fix once and for all a total order on the set of $r$-th roots of unity via setting $\zeta_{k} :=\zeta^{k-1}$ for $1\leq k\leq r.$ Set $S :=\{\zeta_{1},\zeta_{2},\ldots,\zeta_{r}\}.$

A finite dimensional $\widehat{Y}_{r,n}$-module is called calibrated if it has a basis $\{w_{t}\}$ such that for all $1\leq i\leq n$ and all $w_{t},$ we have \[X_{i}w_{t}=\nu_{i}w_{t}\quad \text{for some }\nu_{i}\in \mathbb{C}^{*},\qquad t_{i}w_{t}=\zeta_{j_{i}}w_t \quad \text{for some }\zeta_{j_{i}}\in S.\]

Let $X$ be the abelian group generated by the elements $X_{1}^{\pm1},\ldots,X_{n}^{\pm1}$ of $\widehat{Y}_{r,n}$ and let $V$ denote by all the group homomorphisms from $X$ to $\mathbb{C}^{*},$ which can be identified with $(\mathbb{C}^{*})^{n}$ by identifying the element $\nu=(\nu_1,\ldots,\nu_{n})\in (\mathbb{C}^{*})^{n}$ with the homomorphism given by $\nu(X_{i})=\nu_{i}$ for all $1\leq i\leq n.$

Let $M$ be a finite dimensional $\widehat{Y}_{r,n}$-module. For each $\nu=(\nu_1,\ldots,\nu_{n})\in V$ and $j=(j_{1},\ldots,j_{n})\in J^{n},$ we define the weight spaces and generalized weight spaces of $M$ by
\begin{equation}
M_{(\nu, j)}=\{w\in M\:|\:(X_{a}-\nu_{a})w=(t_{a}-\zeta_{j_{a}})w=0\text{ for all }1\leq a\leq n\}.\label{M-J}
\end{equation}
and
\begin{equation}
M_{(\nu, j)}^{\mathrm{gen}}=\{w\in M\:|\:(X_{a}-\nu_{a})^{N}w=(t_{a}-\zeta_{j_{a}})w=0\text{ for all }1\leq a\leq n\text{ and }N\gg 0\},\label{M-J-J}
\end{equation}
respectively. By definition, we have $M_{(\nu, j)}\subseteq M_{(\nu, j)}^{\mathrm{gen}}$ and $M$ is calibrated if and only if $M_{(\nu, j)}=M_{(\nu, j)}^{\mathrm{gen}}$ for all $\nu\in V$ and $j\in J^{n}.$ Since the elements $X_{i}, t_{i},$ for all $1\leq i\leq n,$ pairwise commute, by Cayley-Hamilton theorem, we have \[M=\bigoplus_{(\nu,j)\in V\times J^{n}}M_{(\nu, j)}^{\mathrm{gen}}.\]
The weights of $M$ are elements of the following set \[\mathrm{supp}(M)=\{(\nu,j)\in V\times J^{n}\:|\:M_{(\nu, j)}^{\mathrm{gen}}\neq 0\}.\]
An element of $M_{(\nu, j)}$ is called a weight vector of weight $(\nu,j).$

Assume that $M$ is a finite dimensional $\widehat{Y}_{r,n}$-module and an element $(\nu,j)=((\nu_1,\ldots,\nu_{n}),$ $(j_{1},\ldots,j_{n}))\in \mathrm{supp}(M).$ For each $1\leq i\leq n-1$ such that $\nu_{i}\neq \nu_{i+1}$ and $j_{i}=j_{i+1},$ we define
\[
  \begin{array}{ccc}
\tau_{i}: ~~~M_{(\nu, j)}^{\mathrm{gen}} &\longrightarrow &M_{(s_{i}\nu, s_{i}j)}^{\mathrm{gen}} \\
w &\longmapsto &\Big(g_{i}-\frac{(q-q^{-1})X_{i+1}}{X_{i+1}-X_{i}}\Big)w;
\end{array}
\]
for each $1\leq i\leq n-1$ such that $j_{i}\neq j_{i+1},$ we define $\tau_{i}w=g_{i}w.$

It is easy to see that the map $\tau_{i}$ is well-defined. The following lemma is proved in [C3, Lemma 2.9].
\begin{lemma}\label{lemma-relations}
For each finite dimensional $\widehat{Y}_{r,n}$-module $M$ and $(\nu,j)\in \mathrm{supp}(M)$, we have
\begin{equation}
\tau_{i}X_{i+1}=X_{i}\tau_{i},\quad X_{i+1}\tau_{i}=\tau_{i}X_{i} \text{ if } \nu_{i}\neq \nu_{i+1}\text{ and } j_{i}=j_{i+1};\label{phi-x-nuj1-1}
\end{equation}
\begin{equation}
\tau_{i}X_{j}=X_{j}\tau_{i}\text{ if } j\neq i, i+1;\label{phi-x-1}
\end{equation}
\begin{equation}
\tau_{i}\tau_{j}=\tau_{j}\tau_{i}\text{ if } \vert i-j\vert> 1;\label{phi-phi-1}
\end{equation}
\begin{equation}
\tau_{i}^{2}=\frac{(qX_{i+1}-q^{-1}X_i)(qX_{i}-q^{-1}X_{i+1})}{(X_{i+1}-X_{i})(X_{i}-X_{i+1})}\text{ if } \nu_{i}\neq \nu_{i+1}\text{ and } j_{i}=j_{i+1};\label{phi-2-nuj1-1}
\end{equation}
\begin{equation}
\tau_{i}\tau_{i+1}\tau_{i}=\tau_{i+1}\tau_{i}\tau_{i+1}\text{ if } \nu_{i},\nu_{i+1}, \nu_{i+2}\text{ are different from each other}.\label{phi-2-nuj1-2}
\end{equation}
\end{lemma}

Let us state the following two facts, which are used in the next section. Assume that $M$ is a finite dimensional $\widehat{Y}_{r,n}$-module and $w_{(\nu,j)}$ is a weight vector in $M$ of weight $(\nu,j).$

\noindent(2.5a) If $\nu_{i}\neq \nu_{i+1}\text{ and } j_{i}=j_{i+1}$, then \[\tau_{i}w_{(\nu,j)}=\Big(g_{i}-\frac{(q-q^{-1})\nu_{i+1}}{\nu_{i+1}-\nu_{i}}\Big)w_{(\nu,j)}\]
is a weight vector of weight $(s_{i}\nu, j).$

\noindent(2.5b) By \eqref{phi-2-nuj1-1}, we have \[\tau_{i}^{2}w_{(\nu,j)}=\frac{(q\nu_{i+1}-q^{-1}\nu_i)(q\nu_{i}-q^{-1}\nu_{i+1})}{(\nu_{i+1}-\nu_{i})(\nu_{i}-\nu_{i+1})}w_{(\nu,j)}.\]
Thus, if $j_{i}=j_{i+1},$ $\nu_{i}\neq \nu_{i+1}$ and $\nu_{i}\neq q^{\pm2}\nu_{i+1},$ then $\tau_{i}w_{(\nu,j)}\neq 0.$

\section{Classification of irreducible calibrated representations}
In this section, we present a complete classification and construction of irreducible calibrated representations of $\widehat{Y}_{r,n}$ over $\mathbb{C}.$
\begin{theorem}\label{classifi-maintheorem}
Let $(\mathrm{c},\bm{\lambda}/\bm{\mu})$ be a placed $r$-skew shape with $n$ boxes. Let $\widehat{Y}^{(\mathrm{c},\bm{\lambda}/\bm{\mu})}$ be the $\mathbb{C}$-vector space spanned by all $w_{\mathcal{T}},$ where $\mathcal{T}\in \mathrm{Std}(\bm{\lambda}/\bm{\mu}).$ We define an action of $\widehat{Y}_{r,n}$ on $\widehat{Y}^{(\mathrm{c},\bm{\lambda}/\bm{\mu})}$ by the following formulas:
\begin{equation}\label{classi-theorem1}
t_{j}w_{\mathcal{T}}=\zeta_{\mathrm{p}(\mathcal{T}|j)}w_{\mathcal{T}}\qquad\text{ for }j=1,2,\ldots,n.
\end{equation}
\begin{equation}\label{classi-theorem2}
X_{j}w_{\mathcal{T}}=q^{2\mathrm{c}(\mathcal{T}|j)}w_{\mathcal{T}}\qquad\text{ for }j=1,2,\ldots,n.
\end{equation}
For $j=1,2,\ldots,n-1,$ if $\mathrm{p}(\mathcal{T}|i)\neq \mathrm{p}(\mathcal{T}|i+1),$ we have
\begin{equation}\label{classi-theorem3}
g_{i}w_{\mathcal{T}}=w_{s_{i}\mathcal{T}};
\end{equation}
if $\mathrm{p}(\mathcal{T}|i)=\mathrm{p}(\mathcal{T}|i+1),$ then
\begin{equation}\label{classi-theorem4}
g_{i}w_{\mathcal{T}}=\frac{q^{2\mathrm{c}(\mathcal{T}|i+1)}(q-q^{-1})}{q^{2\mathrm{c}(\mathcal{T}|i+1)}-q^{2\mathrm{c}(\mathcal{T}|i)}}w_{\mathcal{T}}+
\frac{q^{2\mathrm{c}(\mathcal{T}|i+1)+1}-q^{2\mathrm{c}(\mathcal{T}|i)-1}}{q^{2\mathrm{c}(\mathcal{T}|i+1)}-q^{2\mathrm{c}(\mathcal{T}|i)}}
w_{s_{i}\mathcal{T}},
\end{equation}
where $\sigma(\mathcal{T})$ is the tableau obtained from $\mathcal{T}$ by applying $\sigma$ on the numbers contained in the boxes of $\mathcal{T}$ for any $\sigma\in \mathfrak{S}_{n}.$ Then we have

$(a)$ $\widehat{Y}^{(\mathrm{c},\bm{\lambda}/\bm{\mu})}$ is an irreducible calibrated $\widehat{Y}_{r,n}$-module.

$(b)$ The modules $\widehat{Y}^{(\mathrm{c},\bm{\lambda}/\bm{\mu})}$ for different $(\mathrm{c},\bm{\lambda}/\bm{\mu})$ are non-isomorphic.

$(c)$ Every irreducible calibrated $\widehat{Y}_{r,n}$-module is isomorphic to $\widehat{Y}^{(\mathrm{c},\bm{\lambda}/\bm{\mu})}$ for some placed $r$-skew shape $(\mathrm{c},\bm{\lambda}/\bm{\mu})$.
\end{theorem}
\begin{proof}
The theorem can be proved in exactly the same way as in [Ra1, Theorem 4.1], and we only emphasize the differences.\\

$\clubsuit\hspace{0.5mm}1$  We show that the formulas in (3.1-3.4) define a $\widehat{Y}_{r,n}$-module.

The representations of $\widehat{Y}_{r,2}$ have been studied in [ChPA1, Section 3.4], which implies that the relations $g_{i}^{2}=1+(q-q^{-1})e_{i}g_{i},$ $g_{i}t_{i}=t_{i+1}g_{i}$ and $g_{i}t_{i+1}=t_{i}g_{i}$ hold.

Since $g_{i}w_{\mathcal{T}}$ is a linear combination of $w_{\mathcal{T}}$ and $w_{s_{i}\mathcal{T}},$ the relations $g_{i}t_{j}=t_{j}g_{i},$ for $j\neq i, i+1,$ are also verified. The relations $t_{i}^{r}=1$ and $t_{i}t_j=t_jt_i$ are obviously true.

Then we verify the relations
\begin{equation}\label{classi-relations1-1-11}
X_{i+1}w_{\mathcal{T}}=g_{i}X_{i}g_{i}w_{\mathcal{T}}\quad \text{ for }1\leq i\leq n-1.
\end{equation}
If $\mathrm{p}(\mathcal{T}|i)=\mathrm{p}(\mathcal{T}|i+1)$, it follows from the representation theory of affine Hecke algebras of type $A.$ If $\mathrm{p}(\mathcal{T}|i)\neq\mathrm{p}(\mathcal{T}|i+1)$, then $s_{i}\mathcal{T}$ is standard and we have
\begin{equation}\label{classi-relations1-1-12}
g_{i}X_{i}g_{i}w_{\mathcal{T}}=g_{i}X_{i}w_{s_{i}\mathcal{T}}=
q^{2\mathrm{c}(s_{i}\mathcal{T}|i)}g_{i}w_{s_{i}\mathcal{T}}=
q^{2\mathrm{c}(\mathcal{T}|i+1)}w_{\mathcal{T}}=X_{i+1}w_{\mathcal{T}}.
\end{equation}

Next we verify the relations
\begin{equation}\label{classi-relations1-1}
g_{i}g_{j}w_{\mathcal{T}}=g_{j}g_{i}w_{\mathcal{T}}\quad \text{ for }|i-j|> 1.
\end{equation}
If $\mathrm{p}(\mathcal{T}|i)=\mathrm{p}(\mathcal{T}|i+1)$ and $\mathrm{p}(\mathcal{T}|j)=\mathrm{p}(\mathcal{T}|j+1)$, it follows from the representation theory of affine Hecke algebras of type $A.$

Assume that $\mathrm{p}(\mathcal{T}|i)\neq\mathrm{p}(\mathcal{T}|i+1)$ and $\mathrm{p}(\mathcal{T}|j)\neq\mathrm{p}(\mathcal{T}|j+1)$. Then $s_{i}\mathcal{T},$ $s_{j}\mathcal{T}$, $s_{j}s_{i}\mathcal{T}$ and $s_{i}s_{j}\mathcal{T}$, and we have $s_{i}s_{j}\mathcal{T}=s_{j}s_{i}\mathcal{T}$. Thus, $g_{i}g_{j}w_{\mathcal{T}}=w_{s_{i}s_{j}\mathcal{T}}=w_{s_{j}s_{i}\mathcal{T}}=g_{j}g_{i}w_{\mathcal{T}}.$

Next assume that $\mathrm{p}(\mathcal{T}|i)\neq\mathrm{p}(\mathcal{T}|i+1)$ and $\mathrm{p}(\mathcal{T}|j)=\mathrm{p}(\mathcal{T}|j+1)$. Then $s_{i}\mathcal{T}$ is standard and we have
\begin{equation}\label{classi-relations1-2}
g_{i}g_{j}w_{\mathcal{T}}=\frac{q^{2\mathrm{c}(\mathcal{T}|i+1)}(q-q^{-1})}{q^{2\mathrm{c}(\mathcal{T}|i+1)}-q^{2\mathrm{c}(\mathcal{T}|i)}}w_{s_{i}\mathcal{T}}+
\frac{q^{2\mathrm{c}(\mathcal{T}|i+1)+1}-q^{2\mathrm{c}(\mathcal{T}|i)-1}}{q^{2\mathrm{c}(\mathcal{T}|i+1)}-q^{2\mathrm{c}(\mathcal{T}|i)}}
g_{i}(w_{s_{j}\mathcal{T}})
\end{equation}
and
\begin{equation}\label{classi-relations1-3}
g_{j}g_{i}w_{\mathcal{T}}=\frac{q^{2\mathrm{c}(\mathcal{T}|i+1)}(q-q^{-1})}{q^{2\mathrm{c}(\mathcal{T}|i+1)}-q^{2\mathrm{c}(\mathcal{T}|i)}}w_{s_{i}\mathcal{T}}+
\frac{q^{2\mathrm{c}(\mathcal{T}|i+1)+1}-q^{2\mathrm{c}(\mathcal{T}|i)-1}}{q^{2\mathrm{c}(\mathcal{T}|i+1)}-q^{2\mathrm{c}(\mathcal{T}|i)}}
w_{s_{j}s_{i}\mathcal{T}}.
\end{equation}
If $s_{j}\mathcal{T}$ is standard, then $g_{i}w_{s_{j}\mathcal{T}}=w_{s_{i}s_{j}\mathcal{T}}=w_{s_{j}s_{i}\mathcal{T}};$ if $s_{j}\mathcal{T}$ is not standard, then $s_{j}s_{i}\mathcal{T}=s_{i}s_{j}\mathcal{T}$ is not standard either, and $g_{i}w_{s_{j}\mathcal{T}}=0=w_{s_{j}s_{i}\mathcal{T}}.$ In any case, we have $g_{i}g_{j}w_{\mathcal{T}}=g_{j}g_{i}w_{\mathcal{T}}.$

The case that $\mathrm{p}(\mathcal{T}|i)=\mathrm{p}(\mathcal{T}|i+1)$ and $\mathrm{p}(\mathcal{T}|j)\neq\mathrm{p}(\mathcal{T}|j+1)$ can be dealt with similarly.

Finally we verify the relations
\begin{equation}\label{classi-relations1-4}
g_{i}g_{i+1}g_{i}w_{\mathcal{T}}=g_{i+1}g_{i}g_{i+1}w_{\mathcal{T}}\quad \text{ for }1\leq i\leq n-1.
\end{equation}
If $\mathrm{p}(\mathcal{T}|i)=\mathrm{p}(\mathcal{T}|i+1)=\mathrm{p}(\mathcal{T}|i+2),$ it again follows from the representation theory of affine Hecke algebras of type $A.$

Assume that $\mathrm{p}(\mathcal{T}|i),$ $\mathrm{p}(\mathcal{T}|i+1),$ $\mathrm{p}(\mathcal{T}|i+2)$ are different from each other, then we have
\begin{equation}\label{classi-relations1-5}
g_{i}g_{i+1}g_{i}w_{\mathcal{T}}=w_{s_{i}s_{i+1}s_{i}\mathcal{T}}=w_{s_{i+1}s_{i}s_{i+1}\mathcal{T}}=g_{i+1}g_{i}g_{i+1}w_{\mathcal{T}}.
\end{equation}

Assume that $\mathrm{p}(\mathcal{T}|i)=\mathrm{p}(\mathcal{T}|i+1)\neq\mathrm{p}(\mathcal{T}|i+2).$ Then $s_{i+1}\mathcal{T}$ and $s_{i}s_{i+1}\mathcal{T}$ are standard and we have
\begin{align}\label{classi-relations1-6}
g_{i}g_{i+1}g_{i}w_{\mathcal{T}}&=g_{i}g_{i+1}\Big(\frac{q^{2\mathrm{c}(\mathcal{T}|i+1)}(q-q^{-1})}{q^{2\mathrm{c}(\mathcal{T}|i+1)}-q^{2\mathrm{c}(\mathcal{T}|i)}}w_{\mathcal{T}}+
\frac{q^{2\mathrm{c}(\mathcal{T}|i+1)+1}-q^{2\mathrm{c}(\mathcal{T}|i)-1}}{q^{2\mathrm{c}(\mathcal{T}|i+1)}-q^{2\mathrm{c}(\mathcal{T}|i)}}
w_{s_{i}\mathcal{T}}\Big)\\
&=g_{i}\Big(\frac{q^{2\mathrm{c}(\mathcal{T}|i+1)}(q-q^{-1})}{q^{2\mathrm{c}(\mathcal{T}|i+1)}-q^{2\mathrm{c}(\mathcal{T}|i)}}w_{s_{i+1}\mathcal{T}}+
\frac{q^{2\mathrm{c}(\mathcal{T}|i+1)+1}-q^{2\mathrm{c}(\mathcal{T}|i)-1}}{q^{2\mathrm{c}(\mathcal{T}|i+1)}-q^{2\mathrm{c}(\mathcal{T}|i)}}
g_{i+1}w_{s_{i}\mathcal{T}}\Big)\\
&=\frac{q^{2\mathrm{c}(\mathcal{T}|i+1)}(q-q^{-1})}{q^{2\mathrm{c}(\mathcal{T}|i+1)}-q^{2\mathrm{c}(\mathcal{T}|i)}}w_{s_{i}s_{i+1}\mathcal{T}}+
\frac{q^{2\mathrm{c}(\mathcal{T}|i+1)+1}-q^{2\mathrm{c}(\mathcal{T}|i)-1}}{q^{2\mathrm{c}(\mathcal{T}|i+1)}-q^{2\mathrm{c}(\mathcal{T}|i)}}
g_{i}g_{i+1}w_{s_{i}\mathcal{T}}
\end{align}
and
\begin{align}\label{classi-relations1-7}
g_{i+1}g_{i}g_{i+1}w_{\mathcal{T}}=
\frac{q^{2\mathrm{c}(\mathcal{T}|i+1)}(q-q^{-1})}{q^{2\mathrm{c}(\mathcal{T}|i+1)}-q^{2\mathrm{c}(\mathcal{T}|i)}}w_{s_{i}s_{i+1}\mathcal{T}}+
\frac{q^{2\mathrm{c}(\mathcal{T}|i+1)+1}-q^{2\mathrm{c}(\mathcal{T}|i)-1}}{q^{2\mathrm{c}(\mathcal{T}|i+1)}-q^{2\mathrm{c}(\mathcal{T}|i)}}
w_{s_{i+1}s_{i}s_{i+1}\mathcal{T}}.
\end{align}
If $s_{i}\mathcal{T}$ is standard, then $s_{i+1}s_{i}\mathcal{T}$ and $s_{i}s_{i+1}s_{i}\mathcal{T}$ are standard, and $g_{i}g_{i+1}w_{s_{i}\mathcal{T}}=w_{s_{i}s_{i+1}s_{i}\mathcal{T}}=w_{s_{i+1}s_{i}s_{i+1}\mathcal{T}};$ if $s_{i}\mathcal{T}$ is not standard, then $s_{i+1}s_{i}\mathcal{T}$ and $s_{i}s_{i+1}s_{i}\mathcal{T}$ are not standard either, and $g_{i}g_{i+1}w_{s_{i}\mathcal{T}}=0=w_{s_{i}s_{i+1}s_{i}\mathcal{T}}=w_{s_{i+1}s_{i}s_{i+1}\mathcal{T}}.$ In any case, we have $g_{i}g_{i+1}w_{s_{i}\mathcal{T}}=w_{s_{i+1}s_{i}s_{i+1}\mathcal{T}}.$

The cases that $\mathrm{p}(\mathcal{T}|i)\neq\mathrm{p}(\mathcal{T}|i+1)=\mathrm{p}(\mathcal{T}|i+2)$ and $\mathrm{p}(\mathcal{T}|i)=\mathrm{p}(\mathcal{T}|i+2)\neq\mathrm{p}(\mathcal{T}|i+1)$ can be dealt with similarly. We skip the details. \\

$\clubsuit\hspace{0.5mm}2$  The $\widehat{Y}_{r,n}$-module $\widehat{Y}^{(\mathrm{c},\bm{\lambda}/\bm{\mu})}$ is irreducible.

For each $1\leq k\leq n,$ we define the following two sets:
\[\mathcal{C}(k) :=\{q^{2\mathrm{c}(\mathcal{T}|k)}\:|\:\mathcal{T}\in \text{Std}(\bm{\lambda}/\bm{\mu})\text{ for some }\bm{\lambda}/\bm{\mu}\in \mathcal{S}_{r,n}\},\]
and
\[\overline{\mathcal{C}(k)} :=\{\zeta_{\mathrm{p}(\mathcal{T}|k)}\:|\:\mathcal{T}\in \text{Std}(\bm{\lambda}/\bm{\mu})\text{ for some }\bm{\lambda}/\bm{\mu}\in \mathcal{S}_{r,n}\}.\]

Let $\mathcal{T}$ be a standard tableau of shape $\bm{\lambda}/\bm{\mu}.$ We define
\begin{align}\label{classi-relations1-8}
\pi_{\mathcal{T}}=\prod_{k=1}^{n}\bigg(\prod_{\substack{c\in \mathcal{C}(k)\\c\neq q^{2\mathrm{c}(\mathcal{T}|k)}}}\frac{X_{k}-c}{q^{2\mathrm{c}(\mathcal{T}|k)}-c}\cdot \prod_{\substack{\bar{c}\in \overline{\mathcal{C}(k)}\\\bar{c}\neq \zeta_{\mathrm{p}(\mathcal{T}|k)}}}\frac{t_{k}-\bar{c}}{\zeta_{\mathrm{p}(\mathcal{T}|k)}-\bar{c}}
\bigg).
\end{align}
We have $\pi_{\mathcal{T}}w_{\mathcal{Q}}=\delta_{\mathcal{T}\mathcal{Q}}w_{\mathcal{T}}$ for all $\mathcal{Q}\in \text{Std}(\bm{\lambda}/\bm{\mu})$, which follows from Lemma \ref{lemma-uniques}.

Assume that $M$ is a nonzero submodule of $\widehat{Y}^{(\mathrm{c},\bm{\lambda}/\bm{\mu})}$ and $w=\sum a_{\mathcal{Q}}w_{\mathcal{Q}}$ is a nonzero element of $M$. Let $\mathcal{T}$ be some standard tableau such that the coefficient $a_{\mathcal{T}}$ is nonzero. Then $\pi_{\mathcal{T}}w=a_{\mathcal{T}}w_{\mathcal{T}}$, and so $w_{\mathcal{T}}\in M.$

By Lemma \ref{lemma-bijections}, we can identify $\text{Std}(\bm{\lambda}/\bm{\mu})$ with an interval in $\mathfrak{S}_n.$ Under this identification, there exists a chain $\mathcal{C}<s_{i_1}\mathcal{C}<\cdots<s_{i_{p}}\cdots s_{i_1}\mathcal{C}=\mathcal{T}$ such that all elements of the chain are standard tableaux of shape $\bm{\lambda}/\bm{\mu}.$ By definition, we have
\[\tau_{i_1}\cdots \tau_{i_{p}}w_{\mathcal{T}}=\rho w_{\mathcal{C}}\]
for some $\rho\in \mathbb{C}^{*}.$ So we have $w_{\mathcal{C}}\in M.$

Let $\mathcal{Q}$ be an arbitrary standard tableau of shape $\bm{\lambda}/\bm{\mu}.$ Again, there exists a chain $\mathcal{C}<s_{j_1}\mathcal{C}<\cdots<s_{j_{p}}\cdots s_{j_1}\mathcal{C}=\mathcal{Q}$ of standard tableaux of shape $\bm{\lambda}/\bm{\mu}.$ We have
\[\tau_{j_{p}}\cdots \tau_{j_{1}}w_{\mathcal{C}}=\rho' w_{\mathcal{Q}}\]
for some $\rho'\in \mathbb{C}^{*}.$ So we have $w_{\mathcal{Q}}\in M.$ It follows that $M=\widehat{Y}^{(\mathrm{c},\bm{\lambda}/\bm{\mu})}$ and $\widehat{Y}^{(\mathrm{c},\bm{\lambda}/\bm{\mu})}$ is irreducible. \\

$\clubsuit\hspace{0.5mm}3$  The $\widehat{Y}_{r,n}$-modules $\widehat{Y}^{(\mathrm{c},\bm{\lambda}/\bm{\mu})}$ for different $(\mathrm{c},\bm{\lambda}/\bm{\mu})$ are non-isomorphic.

It follows from Lemma \ref{lemma-uniques}.\\

$\clubsuit\hspace{0.5mm}4$  If $(\nu,j)=((\nu_1,\ldots,\nu_{n}),(j_{1},\ldots,j_{n}))$ is a weight of an irreducible calibrated $\widehat{Y}_{r,n}$-module, then for $1\leq i\leq n,$ we have $v_{i}=q^{2\mathrm{c}(\mathcal{T}|i)}$ and $j_{i}=\zeta_{\mathrm{p}(\mathcal{T}|i)}$ for some standard tableau $\mathcal{T}$ of placed $r$-skew shape.\\

$\clubsuit\hspace{0.5mm}5$  Suppose that $M$ is an irreducible calibrated $\widehat{Y}_{r,n}$-module and that $w_{\nu}$ is a weight vector in $M$ of weight $(\nu,j)$ such that $\nu_{i}=q^{\pm2}\nu_{i+1}$ and $j_{i}=j_{i+1}.$ Then we have $\tau_{i}w_{\nu}=0.$\\

$\clubsuit\hspace{0.5mm}6$  Any irreducible calibrated $\widehat{Y}_{r,n}$-module $M$ is isomorphic to $\widehat{Y}^{(\mathrm{c},\bm{\lambda}/\bm{\mu})}$ for some placed $r$-skew shape $\bm{\lambda}/\bm{\mu}.$

The claims $\clubsuit\hspace{0.5mm}4$-$6$ can be proved in exactly the same way as in [Ra1, Theorem 4.1 Steps 4-6]. We skip the details.
\end{proof}

\section{Applications}
In this section, we shall give several applications of Theorem \ref{classifi-maintheorem}.

\subsection{Restriction and induction}
Let $Y_{r,n}$ be the subalgebra of $\widehat{Y}_{r,n}$ generated by $t_1,\ldots,$ $t_{n},g_{1},\ldots,g_{n-1},$ which is identified the Yokonuma-Hekce algebra. Since $q$ is not a root of unity, $Y_{r,n}$ is semisimple and its irreducible representations are indexed by $r$-partitions of $n.$ The following theorem describes the decomposition of the irreducible $\widehat{Y}_{r,n}$-module $\widehat{Y}^{(\mathrm{c},\bm{\lambda}/\bm{\mu})}$ when it is restricted to the subalgebra $Y_{r,n}.$
\begin{theorem}\label{restriction-theorem-11}
Let $\widehat{Y}^{(\mathrm{c},\bm{\lambda}/\bm{\mu})}$ be the irreducible calibrated representation of $\widehat{Y}_{r,n}$, which is constructed in Theorem \ref{classifi-maintheorem}. Then we have
\begin{equation}\label{classifi-maintheorem-1}
\widehat{Y}^{(\mathrm{c},\bm{\lambda}/\bm{\mu})}\big\downarrow_{Y_{r,n}}^{\widehat{Y}_{r,n}}=\sum_{\bm{\nu}}
c_{\bm{\mu}\bm{\nu}}^{\bm{\lambda}}Y^{\bm{\nu}},
\end{equation}
where $\bm{\nu}$ runs over all $r$-partitions of $n$ and $Y^{\bm{\nu}}$ is the irreducible $Y_{r,n}$-module indexed by $\bm{\nu}.$ Moreover, $c_{\bm{\mu}\bm{\nu}}^{\bm{\lambda}} :=c_{\mu^{(1)}\nu^{(1)}}^{\lambda^{(1)}}\cdots c_{\mu^{(r)}\nu^{(r)}}^{\lambda^{(r)}},$ here each $c_{\mu^{(i)}\nu^{(i)}}^{\lambda^{(i)}}$ is the classical Littlewood-Richardson coefficient.
\end{theorem}

\begin{remark}\label{rem-YH-remark}
The classical Littlewood-Richardson coefficient describes the decomposition of an irreducible representation of $\mathfrak{S}_k\times \mathfrak{S}_l$ when it is induced to $\mathfrak{S}_{k+l}.$ By generalizing this, it has been proved in [St, Theorem 4.5] (also [IJS, Theorem 4.7]) that the $c_{\bm{\mu}\bm{\nu}}^{\bm{\lambda}}$ describes the decomposition of an irreducible representation of $G(r,1,k)\times G(r,1,l)$ when it is induced to $G(r,1,k+l),$ where $G(r,1,k)$ denotes the wreath product $(\mathbb{Z}/r\mathbb{Z})\wr \mathfrak{S}_{k}.$

Theorem \ref{classifi-maintheorem-1} gives a new way of interpreting these coefficients. Then describes the decomposition of the restriction to $Y_{r,n}$ of an irreducible calibrated $\widehat{Y}_{r,n}$-module $\widehat{Y}^{(\mathrm{c},\bm{\lambda}/\bm{\mu})}.$
\end{remark}

Fix $k$ and $l$ such that $k+l=n.$ Let $\widehat{Y}_{r,k}$ be the subalgebra of $\widehat{Y}_{r,n}$ generated by $g_{i},$ $1\leq i\leq k-1,$ $X_{j}^{\pm1}$ and $t_{j},$ $1\leq j\leq k;$ let $\widehat{Y}_{r,l}$ be the subalgebra of $\widehat{Y}_{r,n}$ generated by $g_{i},$ $k+1\leq i\leq n-1,$ $X_{j}^{\pm1}$ and $t_{j},$ $k+1\leq j\leq n.$ In this way, $\widehat{Y}_{r,k}\otimes\widehat{Y}_{r,l}$ is naturally a subalgebra of $\widehat{Y}_{r,n}.$

Let $(\mathrm{a}, \bm{\theta})=(\mathrm{a}, (\theta^{(1)},\ldots,\theta^{(r)}))$ be a placed $r$-skew shape with $k$ boxes and let $(\mathrm{b}, \bm{\phi})=(\mathrm{b}, (\phi^{(1)},\ldots,\phi^{(r)}))$ be a placed $r$-skew shape with $l$ boxes. We label the boxes of $\theta^{(i)}$ with the numbers $k_{i-1}+1,\ldots,k_{i}$ along major diagonals from southwest to northeast for $1\leq i\leq r,$ where $k_{0}=1$ and $k_{i}=|\theta^{(i)}|.$ Similarly, we label the boxes of $\phi^{(i)}$ with the numbers $k+l_{i-1}+1,\ldots,k+l_{i}$ along major diagonals from southwest to northeast for $1\leq i\leq r,$ where $l_{0}=1$ and $l_{i}=|\phi^{(i)}|.$ Let $\widehat{Y}^{(\mathrm{a},\bm{\theta})}$ and $\widehat{Y}^{(\mathrm{b},\bm{\phi})}$ be the irreducible calibrated representations of $\widehat{Y}_{r,k}$ and $\widehat{Y}_{r,l}$ defined in Theorem \ref{classifi-maintheorem}.

In the following we always assume that $(\mathrm{a}, \bm{\theta})$ and $(\mathrm{b}, \bm{\phi})$ are such that
\[\mathrm{a}(\mathrm{box}_{(k_{i})})+1=\mathrm{b}(\mathrm{box}_{(k+l_{i-1}+1)}).\]
Let $\theta^{(i)}\ast_{v}\phi^{(i)}$ (resp. $\theta^{(i)}\ast_{h}\phi^{(i)}$) be the skew shape by placing $\theta^{(i)}$ and $\phi^{(i)}$ adjacent to each other in such a way that $\mathrm{box}_{(k+l_{i-1}+1)}$ of $\phi^{(i)}$ is immediately above (resp. to the right of) $\mathrm{box}_{(k_{i})}$ of $\theta^{(i)}.$

Let $\mathfrak{S}$ be the set of the sequences by permuting the numbers $1,2,\ldots,n.$ For each $0\leq s\leq r$ and $\{i_{1},\ldots,i_{s},j_{1},\ldots,j_{r-s}\}\in \mathfrak{S}$, we define
\[(\bm{\theta}\ast\bm{\phi})_{(i_{1},\ldots,i_{s},j_{1},\ldots,j_{r-s})}=(\theta^{(1)}\ast_{w_{1}}\phi^{(1)},\ldots,\theta^{(r)}\ast_{w_{r}}\phi^{(r)}),\]
where $w_{i_1}=\cdots=w_{i_{s}}=v$ and $w_{j_1}=\cdots=w_{j_{r-s}}=h.$

\begin{theorem}\label{inductions-theorem-11}
With notations and assumptions as above, we have
\begin{equation}\label{classifi-maintheorem-111}
\mathrm{Ind}_{\widehat{Y}_{r,k}\otimes\widehat{Y}_{r,l}}^{\widehat{Y}_{r,n}}(\widehat{Y}^{(\mathrm{a},\bm{\theta})}\otimes \widehat{Y}^{(\mathrm{b},\bm{\phi})})=\sum_{\substack{0\leq s\leq r\\\{i_{1},\ldots,i_{s},j_{1},\ldots,j_{r-s}\}\in \mathfrak{S}}}(\bm{\theta}\ast\bm{\phi})_{(i_{1},\ldots,i_{s},j_{1},\ldots,j_{r-s})}.
\end{equation}
\end{theorem}

\subsection{Simple $Y_{r,n}^{d}$-modules and its center}

Let $d\geq 1$ and $v_1,\ldots,v_d$ be some invertible elements of $\mathbb{C}$. Set $f_{1} :=(X_{1}-v_{1})\cdots (X_{1}-v_{d}).$ Let $\mathcal{J}_{d}$ denote the two-sided ideal of $\widehat{Y}_{r,n}$ generated by $f_{1}.$ We define the cyclotomic Yokonuma-Hecke algebra $Y_{r,n}^{d}=Y_{r,n}^{d}(q)$ to be the quotient $$Y_{r,n}^{d}=\widehat{Y}_{r,n}/\mathcal{J}_{d}.$$

Let $\bm{\lambda}/\bm{\mu}$ be an $r$-skew shape with $n$ boxes. We define
\[\mathrm{NW}(\bm{\lambda}/\bm{\mu})=\{\text{northwest cornor boxes of components of }\bm{\lambda}/\bm{\mu}\}.\]
\begin{lemma}\label{inductions-lemma-111}
Fix $v_1,\ldots,v_d\in \mathbb{C}^{*}$ and let $(\mathrm{c},\bm{\lambda}/\bm{\mu})$ be a placed $r$-skew shape with $n$ boxes. If
\[\{q^{2\mathrm{c}(b)}\:|\:b\in \mathrm{NW}(\bm{\lambda}/\bm{\mu})\}\subseteq \{v_1,\ldots,v_d\},\]
then the $\widehat{Y}_{r,n}$-module $\widehat{Y}^{(\mathrm{c},\bm{\lambda}/\bm{\mu})}$ defined in Theorem \ref{classifi-maintheorem} is a simple $Y_{r,n}^{d}$-module.
\end{lemma}
\begin{proof}
By the formula in Theorem \ref{classifi-maintheorem}, the matrix $\rho(X_1)$ of action of $X_1$ is diagonal with eigenvalues $q^{2\mathrm{c}(\mathcal{T}|1)}$ for $\mathcal{T}\in \text{Std}(\bm{\lambda}/\bm{\mu}).$ The boxes $\mathcal{T}|1,$ $\mathcal{T}\in \text{Std}(\bm{\lambda}/\bm{\mu}),$ are exactly the northwest corner boxes of $\bm{\lambda}/\bm{\mu}.$ Thus, the minimal polynomial of $\rho(X_1)$ is
\[p(t)=\prod_{b\in \mathrm{NW}(\bm{\lambda}/\bm{\mu})}(t-q^{2\mathrm{c}(b)}).\]
We are done.
\end{proof}

\begin{theorem}\label{inductions-theorem-111}
If $Y_{r,n}^{d}$ is semisimple, then its simple modules are exactly the modules $\widehat{Y}^{(\mathrm{c},\bm{\lambda})}$ constructed in Lemma \ref{inductions-lemma-111}, where $\bm{\lambda}=((\lambda_{1}^{(1)},\ldots,\lambda_{d}^{(1)}),\ldots,(\lambda_{1}^{(r)},\ldots,\lambda_{d}^{(r)}))$ is an $r$-tuple of $d$-partitions with $n$ boxes and $\mathrm{c}$ is the content function determined by
\[q^{2\mathrm{c}(b)}=v_{i}\quad \text{if }b\text{ is the northwest corner box of }\lambda_{i}^{(s)}\text{ for any }1\leq s\leq r.\]
\end{theorem}

Let $P_{n}^{\mathbb{C}}(T)$ be the subalgebra of $Y_{r,n}^{d}$ generated by $t_{1},\ldots,t_{n}$ and $X_{1}^{\pm1},\ldots,X_{n}^{\pm1}.$ Set
\[P_{n}^{\mathbb{C}}(T)^{\mathfrak{S}_{n}}\hspace*{-1.5pt}=\hspace*{-1.5pt}\{\sum d_{\alpha,\beta}X^{\alpha}t^{\beta}\hspace*{-2pt}\in \hspace*{-2pt} P_{n}^{\mathbb{C}}(T)|\sum d_{\alpha,\beta}X^{\alpha}t^{\beta}\hspace*{-1.5pt}=\hspace*{-1.5pt}\sum d_{\alpha,\beta}X^{w\alpha}t^{w\beta}~\mathrm{for~any}~w\in \mathfrak{S}_{n}\}.\]
\begin{theorem}\label{inductions-theorem-1111}
If $Y_{r,n}^{d}$ is semisimple, then its center is $P_{n}^{\mathbb{C}}(T)^{\mathfrak{S}_{n}}.$
\end{theorem}
\begin{proof}
It follows from [CWa, Theorem 2.7].
\end{proof}

\subsection{Irreducible calibrated $\widehat{Y}_{r,n}^{p}$-modules}

Let us first recall the presentation of $(\mathbb{Z}/r\mathbb{Z})$-framed affine braid group $(\mathbb{Z}/r\mathbb{Z})\wr B_{n}^{\mathrm{aff}},$ which is given in [ChPA2, (6.5)]. It is generated by $X_{1}, g_{1},\ldots,g_{n-1},t_{1},\ldots,t_{n}$ with the relations:
\begin{equation}\label{rel-def-braidgroup1}\begin{array}{rclcl}
X_{1}g_{0}X_1g_{0}\hspace*{-7pt}&=&\hspace*{-7pt}g_{0}X_1g_{0}X_1,\\[0.1em]
X_1g_{i}\hspace*{-7pt}&=&\hspace*{-7pt}g_{i}X_{1} &&\mbox{for all $i=2,\ldots,n-1$,}  \\[0.1em]
g_ig_j\hspace*{-7pt}&=&\hspace*{-7pt}g_jg_i && \mbox{for all $i,j=1,\ldots,n-1$ such that $\vert i-j\vert \geq 2$,}\\[0.1em]
g_ig_{i+1}g_i\hspace*{-7pt}&=&\hspace*{-7pt}g_{i+1}g_ig_{i+1} && \mbox{for all $i=1,\ldots,n-2$,}\\[0.1em]
t_it_j\hspace*{-7pt}&=&\hspace*{-7pt}t_jt_i &&  \mbox{for all $i,j=1,\ldots,n$,}\\[0.1em]
g_it_j\hspace*{-7pt}&=&\hspace*{-7pt}t_{s_i(j)}g_i && \mbox{for all $i=1,\ldots,n-1$ and $j=1,\ldots,n$,}\\[0.1em]
t_i^r\hspace*{-7pt}&=&\hspace*{-7pt}1 && \mbox{for all $i=1,\ldots,n$,}\\[0.2em]
X_1t_{j}\hspace*{-7pt}&=&\hspace*{-7pt}t_{j}X_{1} && \mbox{for all $j=1,\ldots,n$.}
\end{array}
\end{equation}
We shall denote $(\mathbb{Z}/r\mathbb{Z})\wr B_{n}^{\mathrm{aff}}$ by $\mathcal{B}_{r,1,n}^{\infty}.$ Let $\kappa :\mathcal{B}_{r,1,n}^{\infty}\rightarrow \mathbb{Z}$ is the group homomorphism defined by $\kappa(X_1)=1,$ $\kappa(t_j)=\kappa(g_i)=0$ for $1\leq j\leq n$ and $1\leq i\leq n-1.$ Let $\mathcal{B}_{r,p,n}^{\infty}$ be the subgroup of $\mathcal{B}_{r,1,n}^{\infty}$, which is defined by
\[\mathcal{B}_{r,p,n}^{\infty}=\{b\in \mathcal{B}_{r,p,n}^{\infty}\:|\:\kappa(b)=0~(\mathrm{mod}~p)\}.\]
We define $X_{i+1} :=g_{i}X_{i}g_{i}$ for $1\leq i\leq n-1.$ Set $X^{\varepsilon_{i}}=X_{i}$ for $1\leq i\leq n,$ where $\varepsilon_{i}$ ($1\leq i\leq n$) is a basis of $\mathbb{R}^{n}.$ Set $L :=\sum_{i=1}^{n}\mathbb{Z}\varepsilon_{i}.$ For any $\lambda=\lambda_1\varepsilon_1+\cdots+\lambda_{n}\varepsilon_{n}\in L$, we set $X^{\lambda} :=X_{1}^{\lambda_1}\cdots X_{n}^{\lambda_{n}}.$ Let
\[Q=\sum_{i=2}^{n}\mathbb{Z}(\varepsilon_{i}-\varepsilon_{i-1})\quad\text{ and }\quad L_{p}=Q+\sum_{i=1}^{n}p\mathbb{Z}\varepsilon_{i}.\]
Then $\mathcal{B}_{r,p,n}^{\infty}$ be the subgroup of $\mathcal{B}_{r,1,n}^{\infty}$ generated by $X^{\lambda}$ with $\lambda\in L_{p},$ $g_{i}$ with $1\leq i\leq n-1,$ and $t_{j}$ with $1\leq j\leq n.$

Let $\widehat{Y}_{r,n}^{p}$ be the quotient of the group algebra $\mathbb{C}\mathcal{B}_{r,p,n}^{\infty}$ by the following relation:
\[g_{i}^{2}=1+(q-q^{-1})e_{i}g_{i}\quad\text{ for }1\leq i\leq n-1.\]
Thus, the following set
\[\{t_{1}^{a_1}\cdots t_{n}^{a_n}X^{\lambda}g_{w}\:|\:0\leq a_1,\ldots,a_{n}\leq r-1, \lambda\in L_{p}, w\in \mathfrak{S}_n\}\]
gives a basis of $\widehat{Y}_{r,n}^{p}.$

Let $\zeta=e^{2\pi i/p}$ be a primitive $p$-th root of unity. The following result gives another description of $\widehat{Y}_{r,n}^{p}$, which allows us to apply the Clifford theory developed in [RaR, Appendix].
\begin{theorem}\label{inductions-theorem-113}
The algebra automorphism $\sigma :\widehat{Y}_{r,n}\rightarrow \widehat{Y}_{r,n}$ defined by
\[\sigma(X_{1})=\zeta X_1,\quad \sigma(g_{i})=g_{i}\text{ for }1\leq i\leq n-1,\quad \sigma(t_{j})=t_{j}\text{ for }1\leq j\leq n\]
gives rise to an action of the group $\mathbb{Z}/p\mathbb{Z}=\{1,\sigma,\ldots,\sigma^{p-1}\}$ on $\widehat{Y}_{r,n}.$ Moreover, the set of fixed points of the $\mathbb{Z}/p\mathbb{Z}$-action is exactly $\widehat{Y}_{r,n}^{p}.$
\end{theorem}

If $(\mathrm{c},\bm{\lambda}/\bm{\mu})$ is a placed $r$-skew shape with $n$ boxes and $\sigma\in \mathbb{Z}/p\mathbb{Z}$, we define
\[\sigma(\mathrm{c},\bm{\lambda}/\bm{\mu})=(\mathrm{c}-\alpha i/p,\bm{\lambda}/\bm{\mu}),\]
where $\alpha=2\pi/\mathrm{ln}(q^{2})$ and $\mathrm{c}-\alpha i/p$ denotes the content function defined by $(\mathrm{c}-\alpha i/p)(b)=\mathrm{c}(b)-\alpha i/p$ for all boxes $b\in \bm{\lambda}/\bm{\mu}.$ Let $K_{(\mathrm{c},\bm{\lambda}/\bm{\mu})}$ be the stabilizer of $(\mathrm{c},\bm{\lambda}/\bm{\mu})$ under the action of $\mathbb{Z}/p\mathbb{Z}$, and $k$ be the smallest integer between $1$ and $p$ such that
\[q^{2\sigma^{k}(\mathrm{c})(b)}=q^{2\mathrm{c}(b)}\quad\text{ for all }b\in \bm{\lambda}/\bm{\mu}.\]

Then $K_{(\mathrm{c},\bm{\lambda}/\bm{\mu})}=\{\sigma^{kl}\:|\:0\leq l\leq |K_{(\mathrm{c},\bm{\lambda}/\bm{\mu})}|-1\}.$ We define the element
\begin{equation}\label{calibrated-repre-1}
p_{j}=\sum_{l=0}^{|K_{(\mathrm{c},\bm{\lambda}/\bm{\mu})}|-1}(\zeta^{-jk})^{l}\sigma^{kl}
\end{equation}
is the minimal idempotent of the group algebra $\mathbb{C}K_{(\mathrm{c},\bm{\lambda}/\bm{\mu})}$ corresponding to the irreducible character $\eta_{j},$ where $\eta_{j}$ is given explicitly by
\[\eta_{j} :K_{(\mathrm{c},\bm{\lambda}/\bm{\mu})}\rightarrow \mathbb{C},\quad \sigma^{k}\mapsto \zeta^{jk}.\]

It follows from a standard double centralizer result that, as a $(\widehat{Y}_{r,n}^{p}, K_{(\mathrm{c},\bm{\lambda}/\bm{\mu})})$-bimodule, we have
\begin{equation}\label{calibrated-repre-2}
\widehat{Y}^{(\mathrm{c},\bm{\lambda}/\bm{\mu})}\cong \bigoplus_{j=0}^{|K_{(\mathrm{c},\bm{\lambda}/\bm{\mu})}|-1}\widehat{Y}^{(\mathrm{c},\bm{\lambda}/\bm{\mu},j)}\otimes K^{(j)},
\end{equation}
where $\widehat{Y}^{(\mathrm{c},\bm{\lambda}/\bm{\mu},j)}=p_{j}\widehat{Y}^{(\mathrm{c},\bm{\lambda}/\bm{\mu})}$ and $K^{(j)}$ is the irreducible $K_{(\mathrm{c},\bm{\lambda}/\bm{\mu})}$-module with character $\eta_{j}.$ Applying [RaR, Theorem A.13] we get the following result.

\begin{theorem}\label{inductions-theorem-cal-1114}
Let $(\mathrm{c},\bm{\lambda}/\bm{\mu})$ be a placed $r$-skew shape with $n$ boxes and let $\widehat{Y}^{(\mathrm{c},\bm{\lambda}/\bm{\mu})}$ be the irreducible calibrated $\widehat{Y}_{r,n}$-module, which is constructed in Theorem \ref{classifi-maintheorem}. Let $K_{(\mathrm{c},\bm{\lambda}/\bm{\mu})}$ be the stabilizer of $(\mathrm{c},\bm{\lambda}/\bm{\mu})$ under the action of $\mathbb{Z}/p\mathbb{Z}.$ If $p_{j}$ is the minimal idempotent of $K_{(\mathrm{c},\bm{\lambda}/\bm{\mu})}$ defined in \eqref{calibrated-repre-1}, then $\widehat{Y}^{(\mathrm{c},\bm{\lambda}/\bm{\mu},j)}=p_{j}\widehat{Y}^{(\mathrm{c},\bm{\lambda}/\bm{\mu})}$ is an irreducible calibrated $\widehat{Y}_{r,n}^{p}$-module. Moreover, all the irreducible calibrated $\widehat{Y}_{r,n}^{p}$-modules can be obtained in this way.
\end{theorem}

\subsection{Simple $Y_{r,n}^{d,p}$-modules}
Let us return to $\S4.2$ and consider a particular case. Assume that $p$ divides $d$ and $e=d/p.$ For fixed $x_{0},\ldots,x_{e-1}\in \mathbb{C}^{*}$, we define, for $1\leq j\leq d,$
\[v_{j}=\zeta^{k}x_{l}\quad \text{ if }j-1=lp+k\text{ with unique }0\leq l\leq e-1\text{ and }0\leq k\leq p-1.\]
Then we have \[(X_1-v_1)(X_1-v_{2})\cdots(X_1-v_{d})=(g_{1}^{p}-x_{0}^{p})(g_{1}^{p}-x_{1}^{p})\cdots(g_{1}^{p}-x_{e-1}^{p}).\]

We define the algebra $Y_{r,n}^{d,p}$ as the subalgebra of $Y_{r,n}^{d}$ generated by the elements:
\[a_{0}=X_{1}^{p},\quad a_{1}=X_{1}^{-1}g_{1}X_{1},\quad a_{i}=g_{i-1}\text{ for }2\leq i\leq n,\quad t_{j}\text{ for }1\leq j\leq n.\]

We have a natural surjective homomorphism from $\widehat{Y}_{r,n}^{p}$ to $Y_{r,n}^{d,p}.$ Assume that $\bm{\lambda}=((\lambda_{1}^{(1)},\ldots,\lambda_{d}^{(1)}),\ldots,(\lambda_{1}^{(r)},\ldots,\lambda_{d}^{(r)}))$ is an $r$-tuple of $d$-partitions with $n$ boxes. We shall regard each $d$-partition $(\lambda_{1}^{(h)},\ldots,\lambda_{d}^{(h)})$ as $e$ groups of $p$-partitions, so that we can write
\[\bm{\lambda}=(\lambda_{(l,k)}^{(h)})\quad \text{ for }1\leq h\leq r, 0\leq l\leq e-1\text{ and }0\leq k\leq p-1. \]
Moreover, we define the content of a box $b$ of $\bm{\lambda}$ as \[\mathrm{ct}(b)=\zeta^{k}x_{l}q^{2(j-i)}\quad \text{ if }b\text{ is in position }(i,j)\text{ in }\lambda_{(l,k)}^{(h)}.\]

For each $(r, e, p)$-partition $\bm{\lambda}=(\lambda_{(l,k)}^{(h)})$ and $\sigma\in \mathbb{Z}/p\mathbb{Z}$, we define an action of $\sigma$ on $\bm{\lambda}$ as $\sigma(\bm{\lambda})=(\lambda_{(l,k-1)}^{(h)})$ such that $\mathrm{ct}(\sigma(b))=\zeta^{-1}\mathrm{ct}(b).$ Let $K_{\bm{\lambda}}$ be the stabilizer of $\bm{\lambda}$ under the action of $\mathbb{Z}/p\mathbb{Z}$ and $f_{\bm{\lambda}}$ is the smallest integer between $1$ and $p$ such that $\sigma^{f_{\bm{\lambda}}}(\bm{\lambda})=\bm{\lambda}.$

By [ChPA2, Proposition 4.6], the irreducible representations of $Y_{r,n}^{d}$ are indexed by $r$-tuple of $d$-partitions. For each such $\bm{\lambda},$ we denote by $V^{\bm{\lambda}}$ the corresponding irreducible representations of $Y_{r,n}^{d}.$

It follows from a standard double centralizer result that, as a $(Y_{r,n}^{d}, K_{\bm{\lambda}})$-bimodule, we have
\[V^{\bm{\lambda}}\cong \bigoplus_{j=0}^{|K_{\bm{\lambda}}|-1}V^{\bm{\lambda},j}\otimes Z^{(j)},\]
where $V^{\bm{\lambda},j}$ is a $Y_{r,n}^{d,p}$-module and $Z^{(j)}$ is the irreducible $K_{\bm{\lambda}}$-module with character $\eta_{j},$ which is given explicitly by
\[\eta_{j} :K_{\bm{\lambda}}\rightarrow \mathbb{C},\quad \sigma^{f_{\bm{\lambda}}}\mapsto \zeta^{jf_{\bm{\lambda}}}.\]

Applying [RaR, Theorem A.13] we get the following result.

\begin{theorem}\label{inductions-theorem-cal-1114113}
With notations as above, the modules $V^{\bm{\lambda},j}$, where $\bm{\lambda}$ runs over all $r$-tuple of $d$-partitions and $0\leq j\leq |K_{\bm{\lambda}}|-1,$ form a complete set of non-isomorphic irreducible modules of $Y_{r,n}^{d,p}.$
\end{theorem}

\section{Appendix. Completely splittable representations of degenerate affine Yokonuma-Hecke algebras}

In the appendix, we give a complete classification and construction of irreducible completely splittable representations of degenerate affine Yokonuma-Hecke algebras $D_{r,n}$ and the reflection groups of type $G(r,1,n)$ over an algebraically closed field $\mathbb{K}$ of characteristic $p> 0$ such that $p$ does not divide $r.$ Let $I=\mathbb{Z}\cdot 1\cong \mathbb{Z}_{p}.$

\begin{definition}
The degenerate affine Yokonuma-Hecke algebra, denoted by $D_{r,n}$, is an associative $\mathbb{K}$-algebra generated by the elements $t_{1},\ldots,t_{n},f_{1},\ldots,f_{n-1}, x_1,\ldots, x_{n}$ in which the generators $t_{1},\ldots,t_{n},f_{1},$ $\ldots,f_{n-1}$ satisfy the following relations:
\begin{equation}\label{rel-def-Y1}\begin{array}{rclcl}
f_if_j\hspace*{-7pt}&=&\hspace*{-7pt}f_jf_i && \mbox{for all $i,j=1,\ldots,n-1$ such that $\vert i-j\vert \geq 2$;}\\[0.1em]
f_if_{i+1}f_i\hspace*{-7pt}&=&\hspace*{-7pt}f_{i+1}f_if_{i+1} && \mbox{for all $i=1,\ldots,n-2$;}\\[0.1em]
t_it_j\hspace*{-7pt}&=&\hspace*{-7pt}t_jt_i &&  \mbox{for all $i,j=1,\ldots,n$;}\\[0.1em]
f_it_j\hspace*{-7pt}&=&\hspace*{-7pt}t_{s_i(j)}f_i && \mbox{for all $i=1,\ldots,n-1$ and $j=1,\ldots,n$;}\\[0.1em]
t_i^r\hspace*{-7pt}&=&\hspace*{-7pt}1 && \mbox{for all $i=1,\ldots,n$;}\\[0.2em]
f_{i}^{2}\hspace*{-7pt}&=&\hspace*{-7pt}1 && \mbox{for all $i=1,\ldots,n-1$,}
\end{array}
\end{equation}
together with the following relations concerning the generators $x_1,\ldots, x_{n}$:
\begin{equation}\label{rel-def-Y2}\begin{array}{rclcl}
x_{i}x_{j}\hspace*{-7pt}&=&\hspace*{-7pt}x_{j}x_{i};\\[0.1em]
f_{i}x_{i+1}\hspace*{-7pt}&=&\hspace*{-7pt}x_{i}f_{i}+e_{i};\\[0.1em]
f_{i}x_{j}\hspace*{-7pt}&=&\hspace*{-7pt}x_{j}f_{i} \qquad \qquad \quad\mbox{for all $j\neq i, i+1$;}\\[0.1em]
t_{j}x_{i}\hspace*{-7pt}&=&\hspace*{-7pt}x_{i}t_{j} \hspace{0.07cm}\qquad \qquad\quad\mbox{for all $i, j=1,\ldots,n$.}
\end{array}
\end{equation}
\end{definition}

\begin{remark}\label{rem-degenerate-YH}
The degenerate affine Yokonuma-Hecke algebra $D_{r,n}$ is in fact a special case of the wreath Hecke algebra $\mathcal{H}_{n}(G)$ defined in [WaW, Definition 2.4] when $G=C_{r}$ is the cyclic group of order $r;$ see also [RaS].
\end{remark}

We always consider the category $\mathrm{Rep}_{I}D_{r,n}$ of integral representations of $D_{r,n},$ which consists of the finite dimensional modules on which the generators $x_{i}$'s have eigenvalues lying in $I.$ Recall that $J=\{1,2,\ldots,r\}.$

\begin{definition}
For $M\in \mathrm{Rep}_{I}D_{r,n}$ and $(\alpha,j)=((\alpha_1,\ldots,\alpha_{n}),(j_1,\ldots,j_{n}))\in I^{n}\times J^{n},$ we define the $(\alpha,j)$-weight space of $M$ by
\begin{equation}\label{degenerate-M-J-J}
M_{(\alpha, j)}=\{m\in M\:|\:(X_{i}-\alpha_{i})^{N}m=(t_{i}-\zeta_{j_{i}})m=0\text{ for all }1\leq i\leq n\text{ and }N\gg 0\}.
\end{equation}
\end{definition}
By the Cayley-Hamilton theorem, we have $M=\bigoplus_{(\alpha,j)\in I^{n}\times J^{n}}M_{(\alpha, j)}.$

\begin{definition}
A module $M\in \mathrm{Rep}_{I}D_{r,n}$ is called completely splittable if the $x_{i}$'s act semisimply on it.
\end{definition}

The following lemma follows easily from [WaW, Theorem 5.16] and [Ru, Theorem 2.13].
\begin{lemma}\label{generate-inductions-lemma-1}
Assume that $M\in \mathrm{Rep}_{I}D_{r,n}$ is irreducible and completely splittable. Then for all weights $(\alpha,j)=((\alpha_1,\ldots,\alpha_{n}),(j_1,\ldots,j_{n}))\in I^{n}\times J^{n}$ of $M$ and all $1\leq i< n,$ we have $\alpha_{i}\neq\alpha_{i+1}$ whenever $j_{i}=j_{i+1}.$
\end{lemma}

\begin{proposition}\label{generate-inductions-lemma-2}
Assume that $M\in \mathrm{Rep}_{I}D_{r,n}$ is irreducible and completely splittable. Let $(\alpha,j)=((\alpha_1,\ldots,\alpha_{n}),(j_1,\ldots,j_{n}))\in I^{n}\times J^{n}$ be a weight of $M$, and let $w_{(\alpha,j)}$ be a nonzero weight vector of weight $(\alpha,j)$. Then for all $1\leq i< n,$ we have

$(\mathrm{i})$ If $j_{i}\neq j_{i+1},$ $w_{(s_{i}\alpha,s_{i}j)} :=f_{i}w_{(\alpha,j)}$ is a weight vector of weight $(s_{i}\alpha,s_{i}j)$, and it is not proportional to $w_{(\alpha,j)}.$

$(\mathrm{ii})$ If $j_{i}=j_{i+1}$ and $\alpha_{i+1}=\alpha_{i}\pm 1,$ $f_{i}w_{(\alpha,j)}=\pm w_{(\alpha,j)}.$

$(\mathrm{iii})$ If $j_{i}=j_{i+1}$ and $\alpha_{i+1}\neq\alpha_{i}\pm 1,$ then
\[w_{(s_{i}\alpha,j)} :=\Big(f_{i}-\frac{1}{\alpha_{i+1}-\alpha_{i}}\Big)w_{(\alpha,j)}\]
is a weight vector of weight $(s_{i}\alpha,j).$ The elements $x_{i}, x_{i+1}, f_{i}$ act on the basis $\{w_{(\alpha,j)}, w_{(s_{i}\alpha,j)}\}$ in the following way$:$
\[x_{i}=\left(\hspace{-1.5mm}
  \begin{array}{cc}
    \alpha_{i}  & 0\\
    0   &\alpha_{i+1}
\end{array}
\hspace{-1.5mm}\right),\quad x_{i+1}=\left(\hspace{-1.5mm}
  \begin{array}{cc}
    \alpha_{i+1}  & 0\\
    0   &\alpha_{i}
\end{array}
\hspace{-1.5mm}\right),\quad f_{i}=\left(\hspace{-1.5mm}
  \begin{array}{cc}
    (\alpha_{i+1}-\alpha_{i})^{-1}  & 1-(\alpha_{i+1}-\alpha_{i})^{-2}\\
    1   &(\alpha_{i}-\alpha_{i+1})^{-1}
\end{array}
\hspace{-1.5mm}\right).\]
\end{proposition}

First we consider the case that $p=2.$ In this case, we have $I=\{0,1\}.$ Let $\mathrm{Cont}_{2}(D_{r,n})$ be the set of all the sequences $(\alpha,j)=((\alpha_1,\ldots,\alpha_{n}),(j_1,\ldots,j_{n}))\in I^{n}\times J^{n}$ such that $\alpha_{i}\neq \alpha_{i+1}$ whenever $j_{i}=j_{i+1}.$ We call $s_{i}$ admissible if $j_{i}\neq j_{i+1}.$ We put an equivalence relation $\sim$ on $\mathrm{Cont}_{2}(D_{r,n})$ by saying that $\alpha\sim\beta$ if $\alpha$ can be obtained from $\beta$ by a sequence of admissible transpositions.
\begin{theorem}\label{generate-inductions-theorem-3}
Let $\bm{\lambda}\in \mathrm{Cont}_{2}(D_{r,n})/\sim$ be an equivalence class. Then there exists an irreducible completely splittable module $D^{\bm{\lambda}}$ whose weights are exactly the elements of $\bm{\lambda}.$ $D^{\bm{\lambda}}$ has a basis $\{v_{(\alpha,j)}\:|\:(\alpha,j)\in \bm{\lambda}\}$ with a $D_{r,n}$-action given by
\[t_{i}v_{(\alpha,j)}=\zeta_{j_{i}}v_{(\alpha,j)}\quad \text{ for }1\leq i\leq n,\]
\[x_{i}v_{(\alpha,j)}=\alpha_{i}v_{(\alpha,j)}\quad \text{ for }1\leq i\leq n,\]
If $j_{k}\neq j_{k+1}$ for some $1\leq k\leq n-1$, we have $f_{k}v_{(\alpha,j)}=v_{(s_{k}\alpha,s_{k}j)};$
if $j_{k}=j_{k+1}$, we have $f_{k}v_{(\alpha,j)}=v_{(\alpha,j)}.$ Moreover, all the irreducible completely splittable representations of $D_{r,n}$ can be obtained in this way.
\end{theorem}

From now on, we always assume that $p> 2.$
\begin{definition}
We say a simple transposition $s_{i}$ is admissible with respect to the weight $(\alpha,j)$ if $j_{i}\neq j_{i+1}$ or $\alpha_{i+1}\neq \alpha_{i}\pm 1$ whenever $j_i=j_{i+1}.$
\end{definition}

Let us denote by $\mathrm{Spec}_{p}(D_{r,n})$ the set of all weights occurring in irreducible completely splittable modules of $D_{r,n}.$
\begin{proposition}\label{generate-inductions-lemma-3}
Let $(\alpha,j)\in \mathrm{Spec}_{p}(D_{r,n})$, and suppose that $\alpha_{i}=\alpha_{k}$ for some $1\leq i<k \leq n$ with $j_{i}=j_{k}.$ Then we have 
$\{\alpha_{i}+1, \alpha_{i}-1\}\subseteq \{\alpha_{i_1},\ldots,\alpha_{i_{s}}\}$, where $i+1\leq i_{1},\ldots,i_{s}\leq k-1$ are all the elements such that $j_{i_{1}}=\cdots=j_{i_{s}}=j_{i}=j_{k}.$
\end{proposition}

Let $(\alpha,j)\in I^{n}\times J^{n}$ be such that, whenever $\alpha_{i}=\alpha_{k}$ for some $1\leq i<k \leq n$ with $j_{i}=j_{k},$ we have $\{\alpha_{i}+1, \alpha_{i}-1\}\subseteq \{\alpha_{i_1},\ldots,\alpha_{i_{s}}\}$, where $i+1\leq i_{1},\ldots,i_{s}\leq k-1$ are all the elements such that $j_{i_{1}}=\cdots=j_{i_{s}}=j_{i}=j_{k},$ then we call $(\alpha,j)$ a content vector of length $n.$ We write $\mathrm{Cont}_{p}(D_{r,n})$ for the set of all such vectors, and put an equivalence relation $\sim$ on $\mathrm{Cont}_{p}(D_{r,n})$ by saying that $\alpha\sim\beta$ if $\alpha$ can be obtained from $\beta$ by a sequence of admissible transpositions.
\begin{theorem}\label{generate-inductions-theorem-4}
Let $\bm{\lambda}\in \mathrm{Cont}_{p}(D_{r,n})/\sim$ be an equivalence class. Then there exists an irreducible completely splittable module $D^{\bm{\lambda}}$ whose weights are exactly the elements of $\bm{\lambda}.$ $D^{\bm{\lambda}}$ has a basis $\{v_{(\alpha,j)}\:|\:(\alpha,j)\in \bm{\lambda}\}$ with a $D_{r,n}$-action given by
\[t_{i}v_{(\alpha,j)}=\zeta_{j_{i}}v_{(\alpha,j)}\quad \text{ for }1\leq i\leq n,\]
\[x_{i}v_{(\alpha,j)}=\alpha_{i}v_{(\alpha,j)}\quad \text{ for }1\leq i\leq n,\]
If $j_{k}\neq j_{k+1}$ for some $1\leq k\leq n-1$, we have $f_{k}v_{(\alpha,j)}=v_{(s_{k}\alpha,s_{k}j)};$
if $j_{k}=j_{k+1}$, we have \[f_{k}v_{(\alpha,j)}=(\alpha_{k+1}-\alpha_{k})^{-1}v_{(\alpha,j)}+\sqrt{1-(\alpha_{k+1}-\alpha_{k})^{-2}}v_{(s_{i}\alpha,j)},\]
where we take some fixed choice of square roots in $\mathbb{K}.$ Moreover, all the irreducible completely splittable representations of $D_{r,n}$ can be obtained in this way. 
\end{theorem}

We now apply the results obtained above to the modular representation theory of the wreath product $(\mathbb{Z}/r\mathbb{Z})\wr \mathfrak{S}_{n}.$ Recall that $(\mathbb{Z}/r\mathbb{Z})\wr \mathfrak{S}_{n},$ which is isomorphic to the reflection group $W_{r,n}$ of type $G(r,1,n),$ is generated by elements $t_1,\ldots,t_{n},g_{1},\ldots,g_{n-1}$ with relations:
\begin{equation*}\begin{array}{rclcl}
t_{j}^{r}\hspace*{-7pt}&=&\hspace*{-7pt}1\text{ for }1\leq j\leq n,\qquad \qquad\qquad\qquad \qquad\qquad\quad s_{i}^{2}\hspace*{-7pt}&=&\hspace*{-7pt}1 \text{ for }1\leq i\leq n-1;\\[0.1em]
t_{i}t_{j}\hspace*{-7pt}&=&\hspace*{-7pt}t_{j}t_{i}\text{ for }1\leq i,j\leq n,\qquad \qquad \qquad\qquad\quad s_{i}s_{i+1}s_{i}\hspace*{-7pt}&=&\hspace*{-7pt}s_{i+1}s_{i}s_{i+1} \text{ for }1\leq i\leq n-1;\\[0.1em]
g_{i}t_{j}\hspace*{-7pt}&=&\hspace*{-7pt}t_{s_{i}j}g_{i}\text{ for }1\leq i\leq n-1 \text{ and }1\leq j\leq n,\qquad s_{i}s_{j}\hspace*{-7pt}&=&\hspace*{-7pt}s_{j}s_{i} \text{ if }|i-j|>1.
\end{array}
\end{equation*}

We regard the group algebra $\mathbb{K}W_{r,n}$ of $W_{r,n}$ as the quotient of $D_{r,n}$ obtained by letting $x_{1}=0.$ Thus, for $i\geq 1,$ $x_{i}$ is sent to the $i$-th Jucys-Murphy element $\xi_{k}$ defined by 
\[\xi_{k} :=\sum_{i=1}^{k-1}\sum_{s=0}^{r-1}t_{i}^{s}t_{k}^{-s}(i,k).\]
\begin{definition}
A $\mathbb{K}W_{r,n}$-module $M$ is called completely splittable if the Jucys-Murphy elements $\xi_{k}$'s act semisimply on it.
\end{definition}

\begin{definition}
We define $\mathrm{Spec}_{p}(W_{r,n})\subset\mathrm{Spec}_{p}(D_{r,n})$ to be a set consisting of weights $(\alpha, j)\in \mathrm{Spec}_{p}(D_{r,n})$ which satisfy the additional conditions:

$(1)$ $\alpha_{i_{1}^{p}}=0$ for $1\leq p\leq r,$ where $i_{1}^{p}< i_{2}^{p}<\cdots< i_{s_{p}}^{p}$ are all the elements such that $j_{i_{1}^{p}}=j_{i_{2}^{p}}=\cdots=j_{i_{s_{p}}^{p}}=p.$

$(2)$ For each $i=2,\ldots,n,$ $\{\alpha_{i}+1, \alpha_{i}-1\}\cap \{\alpha_{i_1},\ldots,\alpha_{i_{s}}\}\neq \varnothing$, where $i_1,\ldots,i_{s}\in \{1,\ldots,i-1\}$ are all the elements such that $j_{i_{1}}=\cdots=j_{i_{s}}=j_{i}.$
\end{definition}

\begin{lemma}\label{generate-inductions-lemma-5}
$\mathrm{Spec}_{p}(W_{r,n})$ is exactly the set of weights occurring in irreducible completely spplittable $\mathbb{K}W_{r,n}$-modules.
\end{lemma}

Recall that if $\lambda=(\lambda_{1}\geq\lambda_{2}\geq\cdots\geq \lambda_{k}> 0)$ is a $p$-regular partition of $n,$ we say that $\lambda$ is a splittable shape if $\chi(\lambda)\leq p,$ where $\chi(\lambda)=\lambda_{1}-\lambda_{r}+r.$ An $r$-splittable shape $\bm{\lambda}$ is an $r$-partition of $n$ such that each $\lambda^{(i)}$ is a splittable shape. If $\bm{\lambda}$ is an $r$-splittable shape, we call a standard $\bm{\lambda}$-tableau $\mathfrak{t}$ $p$-standard if for any two of its entries $((i,j),k)$ and $((i',j'),k')$ with $k=k',$ $i> i',$ $j< j',$ and $i-i'+j'-j+1=p,$ we have $\mathfrak{t}((i',j'),k')>\mathfrak{t}((i,j),k),$ where $\mathfrak{t}((i,j),k)$ denotes the entry of $\mathfrak{t}$ in the $i$-th row and $j$-th column of $\mathfrak{t}^{(k)}.$ We have the following result.

\begin{theorem}\label{generate-inductions-theorem-6}
The simple completely splittable $\mathbb{K}W_{r,n}$-modules are indexed by $r$-splittable shapes. The module $D^{\bm{\lambda}}$ corresponding to such a $\bm{\lambda}$ has a basis $\{v_{\mathfrak{t}}\}$ indexed by $p$-standard tableaux of shape $\bm{\lambda}.$ For each nonzero weight vector $v_{\mathfrak{t}}$, we have
\[t_{j}v_{\mathfrak{t}}=\zeta_{\mathrm{p}_{j}}v_{\mathfrak{t}}\quad\text{ for each }1\leq j\leq n;\]

\[g_{i}v_{\mathfrak{t}}=v_{s_{i}\mathfrak{t}}\quad\text{ for each }1\leq i\leq n-1\qquad\text{ if }\mathrm{p}_{i}\neq \mathrm{p}_{i+1};\]

\[g_{i}v_{\mathfrak{t}}=
(\mathrm{c}_{i+1}-\mathrm{c}_{i})^{-1}v_{\mathfrak{t}}+\sqrt{1-(\mathrm{c}_{i+1}-\mathrm{c}_{i})^{-2}}v_{s_{i}\mathfrak{t}}\qquad\text{ if }\mathrm{p}_{i}=\mathrm{p}_{i+1},\]
where $\mathrm{p}_{i}$ is the position of the box containing the number $i$ and $\mathrm{c}_{i}$ is the content of the box containing $i.$
\end{theorem}

\noindent{\bf Acknowledgements.}
The author was partially supported by the National Natural Science Foundation of China (No. 11601273).



\end{document}